\documentclass{amsart}
\usepackage{tikz-cd}
\usepackage{amsmath}
\usepackage{mathtools}
\usepackage{amsmath,amssymb,amsthm,mathrsfs}
\usepackage[utf8]{inputenc}
\usepackage[T1]{fontenc}
\usepackage{hyperref}
\usepackage{enumitem}
\usepackage{verbatim}

\numberwithin{equation}{section}
\theoremstyle{plain} 
\newtheorem{thm}{Theorem}
\numberwithin{thm}{section}
\newtheorem*{thm*}{Theorem}
\newtheorem{prop}[thm]{Proposition}

\theoremstyle{definition}
\newtheorem{definition}[thm]{Definition}

\newtheorem{remark}[thm]{Remark}

\newcommand{\A}{\mathbb{A}}

\newcommand{\C}{\mathbb{C}}

\renewcommand{\H}{\mathbb{H}}

\renewcommand{\P}{\mathbb{P}}
\newcommand{\Q}{\mathbb{Q}}
\newcommand{\R}{\mathbb{R}}

\newcommand{\Z}{\mathbb{Z}}

\newcommand{\cA}{\mathcal{A}}

\newcommand{\cE}{\mathcal{E}}

\newcommand{\cI}{\mathcal{I}}

\newcommand{\cL}{\mathcal{L}}
\newcommand{\cM}{\mathcal{M}}
\newcommand{\cN}{\mathcal{N}}
\newcommand{\cO}{\mathcal{O}}

\newcommand{\bb}{\mathbf{b}}

\newcommand{\bd}{\mathbf{d}}
\newcommand{\be}{\mathbf{e}}

\newcommand{\0}{\mathbf{0}}
\newcommand{\1}{\mathbf{1}}
\newcommand{\2}{\mathbf{2}}

\newcommand{\sD}{\mathscr{D}}

\newcommand{\DR}{\textup{DR}}
\newcommand{\Gr}{\textup{Gr}}

\title{Bott Vanishing via Hodge Theory}
\author{Chuanhao Wei}

\begin{document}
\begin{abstract}
    In this paper, we revise the Bott Vanishing on projective toric varieties by giving it an alternative proof with a condition that is compatible with the condition of Kawamata-Viehweg Vanishing. This proof can also be adapted to generalize Bott Vanishing to the setting using mixed Hodge modules. Lastly, we give a counter-example towards the relative Bott Vanishing for birational morphisms.
\end{abstract}
\maketitle

\section{Introduction}
Let $X$ be a complex projective manifold, with an ample line bundle $\cA$ on it. The renowned Kodaira-Akizuki-Nakano Vanishing Theorem states that 
$$H^q(X, \Omega^p_X\otimes \cA)=0, \text{for } p+q >\dim(X).$$
 If we restrict ourselves to the algebro-geometric setting, there are two directions to generalize this result: one is to weaken the positivity condition on $\cA$ by introducing $\Q$ or ($\R$)-divisors, which is so-called Kawamata-Viehweg Vanishing, e.g. \cite{AMPW}; the other one is by introducing Saito's mixed Hodge modules, \cite{Sch16} for a nice survey. Both generalized vanishing results have some important applications in algebraic geometry. Let me refer to \cite{Wei23} for a detailed discussion on this topic, and making a further generalization using Sabbah-Mochizuki's mixed Twistor $\sD$-modules. 

In one special case, which is when $X$ is a smooth projective toric variety, the above vanishings have a wider range. Actually, 
$$H^k(X, \Omega^p_X\otimes \cA)=0, \text{for } k>0.$$
In this case, it is called Bott Vanishing. Let me refer to \cite{Fuj07} and \cite{Fuj08} for the history of it and an algebraic proof using the multiplication map on $X$ to reduce it to Serre Vanishing. One natural question to ask is that if such vanishing works for a larger class of varieties. \cite{Tot20} and \cite{Tot23} study the cases for surfaces and Fano-3-folds. In this paper, we essentially still focus on the toric variety case, and we try to generalize it in the two directions as we mentioned above, that were used to generalize Kodaira-Akizuki-Nakano Vanishing. Please note that in Fujino's papers, he also has a version of Bott Vanishing for singular toric varieties. However, the method this this paper cannot be directly applied to the singular case. Please see Remark \ref{R: singular}, for more details.

Let $X$ be a smooth projective toric variety, containing the algebraic torus $T\simeq G_m^r$, with $D=X\setminus T$, the torus-invariant simple normal crossing (SNC) divisor. We have the decomposition of irreducible components $D=D_1+...+D_n$. For a vector $\bd=(d_1,d_2,...,d_n)\in \R^n$, we denote $\bd D= d_1D_1+...+d_nD_n$. We say $\bd \leq \be,$ for some $\be \in \R^n$, if $d_j\leq e_j, $ for all $j$, and we denote $\0=(0,...,0)$, similarly for $\1$, etc., with the size being clear from the context. 

Now, we state our first vanishing result.
\begin{thm}\label{T: main 1}
    In the above setting, let $D'$ be a divisor that is part of $D$, and $\cL$ be a line bundle on $X$. If $\cL (-\bd D')$ is ample as a $\R$-divisor, for some $\0\leq \bd\leq \1$, then we have the Bott-type vanishings:
$$H^k(X, \Omega^p_X(\log D')(-D')\otimes \cL)=0,
$$
for all $p\geq 0$ and $k\geq 1$.
\end{thm}

It is suggested to compare this result to \cite[Theorem 1.1]{Fuj07}, and to \cite[Theorem 2.1.1]{AMPW}. The setting above is more compatible with the Kawamata-Viehweg type generalization.

Let's consider three extremal cases of the theorem. The first case is that when $D'=\emptyset$, it is the classical Bott vanishing. The second case is that when $D'=D$, by noticing that $\Omega^p_X(\log D)$ is a trivial vector bundle on $X$, we can reduce it to the classical Kawamata-Viehweg vanishing, which implies $H^k(X, \omega_X\otimes \cL)=0,$ for $k>0$. Actually, the idea of the proof the theorem in general is to reduce it to such case. The third case is that when $\cL\simeq \cO_X(D')$, then we are requiring the $D'$ containing the support of an ample torus invariant divisor, which is equivalent to ask $U=X \setminus D'$ being contained in one affine chart of the toric variety $X$. (This is because, otherwise, we can construct a toric inclusion $\P^1\to U$.)  Hence $U$ is isomorphic to the affine space $\A^r$ removing several coordinate hyperplanes. Recall that the mixed Hodge structure on $V^*=H^*(U, \C)$ is induced by the hyper-cohomologies of the log-de Rham complex $\H^*(\Omega^\bullet_X(\log D'))$, with the so-called stupid filtration to give the Hodge filtration on $V^*$, and we get $V^k\simeq \bigoplus_{p+q=k}H^q(X, \Omega^p_X(\log D')) $. However, in our setting, it can be computed straight forwardly that $\dim V^k= \dim H^0(X, \Omega^k_X(\log D')).$ This implies the vanishings we want, and provides an evidence that such vanishings may also related to Hodge theory. 

Actually, we can have a much more general Bott-type Vanishing in the setting of $\C$-mixed Hodge modules, as generalizing Kodaira-Akizuki-Nakano Vanishing to Saito Vanishing. However, to achieve the broader vanishing range comparing to Saito Vanishing, it seems to the author that we are obliged to only consider the mixed Hodge modules that have nice restriction properties along those torus invariant strata. Precisely, we need to require the mixed Hodge module $\cM$ is non-characteristic with respect to all strata of $D$, and we call it being \emph{non-characteristic with respect to $(X, D)$}. In this paper, we always use a \emph{right} $\sD$-module with an increasing good (Hodge) filtration $(\cM, F_\bullet)$ to represent a $\C$-mixed Hodge module, omitting its weight filtration and assuming its polarizability, and for the most of the time, we may further omit the filtration $F_\bullet$, saying that $\cM$ underlies a mixed Hodge module. We use $\Gr^F_\bullet \cM$ to denote the associated graded pieces of $\cM$ with respect to the filtration $F_\bullet$. Since we use right $\sD$-modules, it is natural to consider $T^p_X(-\log D')$, the natural dual bundle of $\Omega^p_X(\log D')$.

\begin{thm}\label{T: main 2}
    In the same setting as in the previous theorem, assume that we have a $\C$-mixed Hodge module $(\cM, F_\bullet)$ on $X$ that is non-characteristic with respect to $(X, D)$. Let $D'$ be a divisor that is part of $D$, and $\cL$ be a line bundle on $X$.  If $\cL (-\bd D')$ is ample as a $\R$-divisor, for some $\0\leq \bd\leq \1$, then we have the Bott-type vanishings:
$$H^k(X, T^p_X(-\log D')\otimes \Gr^F_i\cM \otimes \cL)=0,
$$
for all $i\in \Z$, $p\geq 0$ and $k\geq 1$.
\end{thm}

Moreover, according to the proof of both theorems that will appear in the next section, the reason that we need $X$, assuming smoothness and projectivity, to be a toric variety is only due to the fact that the sheaf of log-one-forms $\Omega^1_X(\log D)$ is a trivial bundle on $X$. This means that Bott-type vanishing holds for any projective smooth log-pair with SNC boundary $(X, D)$, with $\Omega^1_X(\log D)$ being a trivial bundle. By considering the quasi-albanese map of $X\setminus D$, \cite{Fu14}, such condition actually implies that $X\setminus D$ is a quasi-abelian variety (in the sense of Iitaka), with a log-smooth toric-bundle compactification. In this paper, we do not use the group structure on $X\setminus D$.

\begin{thm}\label{T: main}
    Theorem \ref{T: main 1} and more generally, Theorem \ref{T: main 2} still holds if $(X, D)$ is a smooth toric-bundle compactification of a quasi-abelian variety $X\setminus D$, with $D$ being the boundary divisor.
\end{thm}

In \S3, we will give a counter-example of the relative Bott vanishing that is conjectured in \cite{Tot23}. To get the non-vanishing, the general idea shares some similarity with the method of the proof of the Bott-type vanishings above, but bears more quantitative computations. 

\section{The proof of the Bott-type vanishings}
Let $(X, D)$ be a projective log-pair with $X$ smooth and $D$ SNC. We use $D^i_j$, $i\in{1,...,r}$ and $j\in \Z^+$, to denote the irreducible components of the $i$-th order strata of $D$, i.e. each $D^i_j$ has codimension $i$ in $X$ and can realized an intersection of $i$ different irreducible components of $D$. We have the decomposition of irreducible components $D=D_1+...+D_n$, with $D_j=D^1_j.$  Each $D^i_j$ with the naturally induced boundary divisor, defined as $\sum_{\{k|D^i_j \nsubseteq D_k\}}D_k|_{D^i_j}$, is also a projective log-smooth pair.

\begin{definition}\label{D: D-reg}
    Let $\cE$ be a quasi-coherent sheaf on $X$, and we say that it is \emph{$D$-regular}, if locally, a sequence of functions that defines each of the irreducible components of $D$ forms a regular sequence on $\cE$. 
\end{definition}

One example for a \emph{$D$-regular} sheaf is any locally free sheaf. Another less trivial example is the following. Assume we have a $\C$-mixed Hodge module $\cM$ on $X$ and it is non-characteristic with respect to $(X, D)$, i.e. all inclusions $D^i_j\to X$ are non-characteristic with respect to $\cM$. Then $\cM$ is $D$-regular. Hence, any of its subsheaf, or further tensor with a locally free sheaf, is also $D$-regular. Be careful that, if we use left $\sD$-modules, $\cM^i_j$ is still a $\sD$-module in the previous assumption. However, in this paper, we only use right $\sD$-modules to make it easy to be compared with different generalized Saito Vanishing, so to make it is still a $\sD$-module, we own a small modification, which will be addressed in Proposition \ref{P: restriction of right D-mod}

As mentioned in the Introduction, we will consider the case that $(X, D)$ satisfying that $\Omega^1_X(\log D)$ is a  trivial bundle. Note that $(D^i_j, \sum_{\{k|D^i_j \nsubseteq D_k\}}D_k|_{D^i_j})$ is  also a projective log-smooth pair with the sheaf of log-one-forms being a trivial bundle. To see this, (without using the fact that $(X, D)$ is a quasi-abelian variety,) let's only consider the $D_1$ case, and denote the induced boundary divisor $(D- D_1)|_{D_1}$. Then, the claim is due to the following short exact sequence:
\begin{equation}\label{E: the exact sequence}
    0\to \cO_{D_1}\to  \Omega^1_X(\log D)|_{D_1} \to \Omega^1_{D_1}(\log (D- D_1)|_{D_1})\to 0,
\end{equation}
where the first map is locally defined by sending $1$ to $dt/t$, $t$ a local defining function of $D_1\subset X$.

\begin{prop}\label{P: Induction}
In the above setting. Let $D'$ be a divisor that is part of $D$. If we have a $D$-regular quasi-coherent sheaf $\cE$, such that 
$$H^k(X, \omega_X\otimes \cE)=0 \text{ and } H^k(D^i_j, \omega_{D^i_j}\otimes \cE|_{D^i_j})=0,$$
for all such $i,j$ that $D^i_j$ is not contained in $D'$, and $k\geq 1$, then we have 
$$H^k(X, \Omega^p_X(\log D')(-D')\otimes \cE)=0,
$$
for all $p$ and $k\geq 1$.

\end{prop}
\begin{proof}
Note that, in the case that $D'=D$, the proposition is trivial by noticing that $\Omega^p_X(\log D)$ are trivial vector bundles. We can conclude the proof of the proposition by arguing inductively on the number of components of $D-D'$ and the dimension of $X$, by adding one component $H\subset D$ (, which is not part of $D'$, ) onto $D'$ and considering the following short exact sequence:
\begin{align}
    0\to  \Omega^p_X(\log (D'+H))(-D'-H)\otimes \cE &\to \Omega^p_X(\log D')(-D')\otimes \cE \label{E: 1}  
    \\&\to \Omega^p_{H}(\log D'|_{H})(-D'|_{H})\otimes \cE|_{H}\to 0.\nonumber
\end{align}
\end{proof}

\begin{proof}[Proof of Theorem \ref{T: main 1} (in the setting of Theorem \ref{T: main})]
    Due to the previous proposition, we are reduced to check the following case: 
    \begin{equation*}
        H^k(X, \omega_X\otimes \cL)=0, \text{for } k>0
    \end{equation*}
    This just Kawamata-Viehweg Vanishing, with SNC boundaries, e.g. \cite[\S 6.2.b)]{EV92}.
\end{proof}

\begin{remark}
    If $D=A+B$ is a reduced SNC divisor on $X$ separated into two parts, we use $\C_X[*A+!B][n]$ to denote the perverse sheaf that is constructed by the constant perverse sheaf on $X\setminus (A+B)$, and using $Ri_*$ functor extending along $A$, and $Ri_!$ along $B$(, which does not depend on the order of the extension by Deligne's logarithmic comparison).  Using $\cE=\cO_X$, the exact sequences (\ref{E: 1}) in the previous proposition can be related to the exact sequences in the abelian category of perverse sheaves on $X$:
    \begin{align*}
            0\to \C_{H}[*(D'-D'')|_{H}+!D''|_{H}][n-1]\to& \C_X[*(D'-D'')+!(D''+H)][n] \\
            \to&\C_X[*(D'-D'')+!D''][n]\to 0.
    \end{align*}
by lifting it as a sequence of mixed Hodge modules, and considering their corresponding log-de Rham complexes, \cite{W17a}. More generally, if we take $\cE$ to be a Hodge bundle corresponding to a graded polarizable variation of mixed Hodge structures, then we just replace the constant local system $\C$ above by the local system that corresponds to it. To prove Theorem \ref{T: main 2}, we need to work in the setting of Saito's mixed Hodge modules, following a similar strategy.
\end{remark}

\begin{remark}\label{R: singular}
    The main obstruction for using the current argument to prove the Bott Vanishing on singular toric varieties is that the short exact sequence in the previous remark does not hold for normal toric varieties in general. However, it still holds in the case of simplicial toric varieties, hence Theorem \ref{T: main 1} still holds using essentially the same proof, by replacing those sheaves of log-forms by their corresponding reflexive sheaves. For the general normal toric varieties case, we need to have a better understanding of the simplicial resolution of the toric variety, \cite{Ish87}. 
\end{remark}

Before we move on to the proof of Theorem  \ref{T: main}, we recall some basic properties of the mixed Hodge module $\cM$ in the setting of Theorem  \ref{T: main}. 
\begin{prop}\label{P: restriction of right D-mod}
    Let $\cM$ be a mixed Hodge module on a smooth variety $X$, and $D$ a SNC divisor on $X$. Let $H$ be any irreducible component of $D$, and $i_H: H\to X$ be the natural inclusion.   Assume that $\cM$ is non-characteristic with respect to $(X, D)$. Then,  we have  
    $$i_H^*\cM[1]\simeq \cM_H:= \cM(H)|_H(=\cM (H)/\cM), $$
    which still underlies a mixed Hodge module on $H$.  Furthermore, setting $D_H=(D-H)|_H$, $\cM_H$ is non-characteristic with respect to $(H, D_H)$.
\end{prop}
\begin{proof}
    The first statement is due to \cite[3.5.6. Lemme]{Sa88}, or \cite[2.25. Lemma]{Sa90}. The second statement is due to \cite[3.5.4. Lemme]{Sa88}.
\end{proof}
Using the proposition, we get a short exact sequence in the abelian category of mixed Hodge modules on $X$ by
$$0\to i_+\cM_H
\to \cM[!H]
\to  \cM \to 0.
$$
As in \cite[3.5.6 Lemme]{Sa88}, we can explicitly describe the $\R$-indexed Kashiwara-Malgrange (KM) filtration $V^H_\bullet \cM(*H)$: 
\begin{equation}\label{E: V resp to H}
V^H_a \cM= \cM(\lfloor a +1\rfloor H).    
\end{equation}
Moreover, we have an explicit formula for the multi-indexed KM filtration with respect to $D$. Recall that the multi-indexed KM filtration $V^D_\bullet$ is defined by taking the intersection of the corresponding KM filtration with respect to each irreducible component. It only behaves well under the condition of $V$-compatibility, see \cite{W17a} and \cite[\S3]{Wei23} for more details.

\begin{prop}\label{P: multi-V for M(*D)}
    In the same setting as in the previous proposition, we have 
    \begin{equation}\label{E: V resp to D}
    V^D_{\bd}\cM(*D)=  \cM(\lfloor \bd +\1\rfloor D).    
    \end{equation}
    In particular, $\cM(*D)$ is $V$-compatible with respect to $D$.
\end{prop}

\begin{proof}
    Using (\ref{E: V resp to H}), we just need to consider the $\cO_X(*D)$-module structure on $\cM(*D)$. 
    Working locally around a point $x\in X$, and we can assume that all components of $D=D_1+...+D_k$ contains $x$. Denote $t_i$ is a local function that defines $D_i$, then the set $\{t_i\}$ forms a regular sequence on $\cM(*D)$. Moreover, $t_i$ are invertible in $\cO_X(*D)$, we just need to show (\ref{E: V resp to D}) for $-\2 \leq \bd \leq -\1$. Since $V^D_{-\1}\cM(*D)=\cM$, by induction on the number of components of $D$, we just need to show $\bigcap_i \cM(-D_i)=\cM(-D).$ By induction again, we just need to show $\cM(-D_1)\cap \cM(-D+D_1)=\cM(-D).$ The $\supset$ part is obvious, so we argue the $\subset$ part. Let $m\in \cM(-D_1)\cap \cM(-D+D_1)$, so we can write $m=t_2\cdot...\cdot t_k\cdot n$, for some $n\in \cM$. Since $\{t_2,.., t_k\}$ forms a regular sequence on $\cM/\cM(-D_1)$, we get that $n\in \cM(-D_1)$, which implies $m\in \cM(-D)$.
\end{proof}

Now, we are ready to give:

\begin{proof}[Proof of Theorem  \ref{T: main 2} (in the setting of Theorem \ref{T: main})]
    Recall that we have a $\C$-mixed Hodge module $\cM$ on $X$ satisfying the conditions in the previous two propositions. $\cL$ is a line bundle on $X$ satisfying $\cL (-\bd D')$ is ample as a $\R$-divisor, for some $\0\leq \bd\leq \1$. 
    
    We first consider the case that $D=D'$, and since $T^p_X(-\log D)$ is a trivial bundle, we just need to show that 
    \begin{equation}
        H^k(X, \Gr^F_i\cM \otimes \cL)=0, \text{for } i\in \Z, k\geq 1. \label{E: D=D' case}
    \end{equation}
    Note that, in the case that $\Gr^F_i \cM$ is the lowest graded piece, it is just the Kawamata-Viehweg vanishing in the setting of mixed Hodge modules, e.g. \cite{Suh18}, \cite{Wu22}, which need not the toric assumptions. For the general pieces, we can first uniquely write $\bd D=\bb B+ C$, with $D=B+C$ and  $\0\leq \bb <\1$. Due to Proposition \ref{P: multi-V for M(*D)}, we have $V^B_{<-\bb}V^C_{-\1}\cM(*D)\simeq \cM$.  Apply \cite[Theorem 1.2]{Wei23}, we get the vanishings for the twisted graded de Rham complex:
    $$\H^k(\DR_{(X,D)}\Gr^F_\bullet \cM^{}(*D)\otimes \cL)=0, \text{for } k>0.
    $$
    Using the fact that $T^p_X(-\log D)$ being trivial, we can conclude (\ref{E: D=D' case}) by induction on the grading of $\Gr^F_\bullet$. Such argument can be found in \cite[Lemma 2.5]{PS13} and \cite[Corollary 21]{W17b}. 

    Then, we want to apply the induction as we did for the proof of Theorem \ref{T: main 1}. Let $H$ be an irreducible component of $D-D'$. Using (\ref{E: 1}) with $\cE=\omega_X^{-1}$, we have the short exact sequence:
    $$0\to T^p_X(-\log (D'+H))\to T^p_X(-\log D')\to T^p_H(-\log D_H)\otimes \cO_X(H)|_H\to 0,
    $$
    which induces the following short exact sequence,
    $$0\to T^p_X(-\log (D'+H))\Gr^F_\bullet \cM\to T^p_X(-\log D')\Gr^F_\bullet \cM \to T^p_H(-\log D_H)\Gr^F_\bullet \cM_H\to 0.
    $$
    Twisting it by $\cL$, we can conclude the proof by induction on the number of components of $D-D'$ and the dimension of $X$.
\end{proof}

\section{A counter-example for relative Bott Vanishing for birational maps}
This section is dedicated to give a counter-example to \cite[Conjecture 4.1]{Tot23}.

Let $Z=\A^3$. Let $g:Y \to Z$ be the blowup at the origin, with the exceptional divisor $P\simeq \P^2$ on $Y$. Let $f:X \to Y$ be the blowup of a smooth curve $F$ on $P$ of degree $d$, and denote the exceptional divisor on $X$ by $E$, which is a $\P^1$ bundle over $F$ that is isomorphic to $\P(\cN^*_{F/Y})$.

To understand the curve classes on $E$, we consider the short exact sequence 
\begin{equation}\label{E:SES for N}
0\to \cN^*_{P/Y}|_F\to\cN^*_{F/Y}\to \cN^*_{F/P}\to 0.
\end{equation}
(It can be induced by $\cN^*_{F/Y}\to \Omega^1_Y |_F\to \Omega^1_F$, and its counterpart for $F/P$.) We get 
\begin{equation}\label{E: deg of normal bund}
    \deg(\wedge^2 \cN^*_{F/Y})=\deg(\cN^*_{P/Y}|_F)+\deg(\cN^*_{F/P})=d-d^2.
\end{equation}
 By twisting $(\cN^*_{P/Y}|_F)^{-1}$ on (\ref{E:SES for N}), we get that $e=d^2+d$ for $E$, using the notation in \cite[V Notation 2.8.1.]{Har77}. This allows us to apply \cite[V Prop 2.20]{Har77}.
Let $D=-E-f^*(d+1)P$, and denote $a$ and $b$ the curve classes that are contracted by $f$ and $g$ respectively. We have $a\cdot D=b\cdot D=1$. According to \emph{op. cit.}, this implies $D$ is relatively ample. We want to argue that, when $d$ is sufficient large,
\begin{equation}\label{E: main}
R^1(g\circ f)_*(\Omega^2_X(-E-f^*(d+1)P))\neq 0, 
\end{equation}
which implies relative Bott Vanishing fails.

Since relative Bott Vanishing holds for a blowup along a smooth locus, \cite[Corollary 4.4]{Tot23}, we have $R^i f_*\Omega^2_X(-E)=0$, for $i>0$. Hence, by projection formula and Leray spectral sequence, to get (\ref{E: main}), we only need to show
\begin{equation}\label{E:fir}
    R^1 g_*(f_*(\Omega^2_X(-E))(-(d+1)P))\neq 0
\end{equation}

We have a short exact sequence, (Corollary 6.2 of Totaro's paper)
$$0\to f_*\Omega^2_X(-E) \to \Omega^2_Y \to \Omega^1_F\otimes \cN^*_{F/Y}\to 0.$$
Combining it with the short exact sequence 
$$0\to \wedge^2 \cN^*_{F/Y}\to \Omega^2_Y |_F\to \Omega^1_F\otimes \cN^*_{F/Y}\to 0,
$$
(which is induced by the short exact sequence $0\to \cN^*_{F/Y}\to \Omega^1_Y |_F\to \Omega^1_F\to 0$,)
and 
$$0\to \Omega^2_Y\otimes \cI_F\to\Omega^2_Y\to \Omega^2_Y |_F\to 0,
$$
and using snake lemma, we have a short exact sequence
$$0\to \Omega^2_Y\otimes \cI_F\to f_*\Omega^2_X(-E) \to \wedge^2 \cN^*_{F/Y}\to 0.
$$
Recall (\ref{E: deg of normal bund}), $\deg(\wedge^2 \cN^*_{F/Y})=d-d^2$. Set $A$ a divisor on $F$ of degree $d(d+1)$, such that its corresponding line bundle is isomorphic to $\cO_Y(-(d+1)P)|_F\simeq \cO_P(d+1)|_F$. Then, set $\cL= \wedge^2 \cN^*_{F/Y}(A)$, and $\deg \cL= 2d$. By Serre duality, $h^1(F, \cL)=h^0(F, \omega_F\otimes \cL^{-1})$. By Riemann-Roch, we have 
$$h^0(F, \omega_F\otimes \cL^{-1})=h^0(F, \cL)-(2d-g+1),$$
Using genus-degree formula, 
$$h^0(F, \omega_F\otimes \cL^{-1})\geq -2d+ \frac{1}{2}(d-1)(d-2)-1>0,$$
when $d\geq 8$. In such case, we get 
$$H^1(F, \wedge^2 \cN^*_{F/Y}(A))\neq 0.$$

Now, to get (\ref{E:fir}), it is suffice to show that 
\begin{equation}\label{E:sec}
    R^2g_*(\Omega^2_Y(-(d+1)P)\otimes \cI_F)=0.
\end{equation}
Consider the short exact sequence 
$$0\to \Omega^2_Y(-(d+2)P) \to \Omega^2_Y(-(d+1)P)\otimes \cI_F \to \Omega^2_Y(-(d+1)P)\otimes \cO_P(-F)\to 0.
$$
It is induced by $0\to \cO_Y(-P)\to \cI_F \to \cO_P(-F)\to 0$.
Due to relative Bott vanishing, (or just relative Nakano vanishing in this case,) $R^2g_*(\Omega^2_Y(-(d+2)P))=0$, and note that $\Omega^2_Y(-(d+1)P)\otimes \cO_P(-F)\simeq \Omega^2_Y|_P\otimes \cO_P(1)$, 
hence to achieve (\ref{E:sec}), it is suffice to show 
\begin{equation}\label{E:thi}
    R^2g_*(\Omega^2_Y|_P\otimes \cO_P(1))=0
\end{equation}

Use the short exact sequence
$$ 0\to \Omega^2_Y(-2P) \to \Omega^2_Y(-P) \to \Omega^2_Y|_P\otimes \cO_P(1)\to 0,
$$
and relative Bott vanishing, we get (\ref{E:thi})

For completeness, let's also show, for $i\geq 1$
\begin{equation}\label{E:main'}
    R^i(g\circ f)_*(\Omega^1_X(-E-f^*(d+1)P))=0, 
\end{equation}
Due to relative Nakano vanishing and projection formula, it is equivalent to 
\begin{equation}\label{E:fir'}
    R^i g_*(f_*(\Omega^1_X(-E))(-(d+1)P))=0.
\end{equation}
Due to (Corollary 6.2 of Totaro's paper), we have
$$ f_*\Omega^1_X(-E) \simeq\Omega^1_Y\otimes\cI_F,
$$
so we just need to show 
$$R^i g_*(\Omega^1_Y(-(d+1)P)\otimes\cI_F)=0
$$
Then, it follows via an essentially same argument as the proof for (\ref{E:sec}).

\bibliographystyle{alpha}
\bibliography{bib}

\end{document}